\documentclass[12pt, reqno]{amsart}
\UseRawInputEncoding
\usepackage{amsmath, amsthm, amscd, amsfonts, amssymb, graphicx, color}
\usepackage[bookmarksnumbered, colorlinks, plainpages]{hyperref}
\input{mathrsfs.sty}

\hypersetup{colorlinks=true,linkcolor=red, anchorcolor=green,
citecolor=cyan, urlcolor=red, filecolor=magenta, pdftoolbar=true}

\newtheorem{theorem}{Theorem}[section]
\newtheorem{lemma}[theorem]{Lemma}

\newtheorem{corollary}[theorem]{Corollary}
\theoremstyle{definition}

\newtheorem{example}[theorem]{Example}

\theoremstyle{remark}

\numberwithin{equation}{section}

\newcommand{\Ka}{\mathcal{K}}
\newcommand{\Ba}{\mathcal{B}}
\newcommand{\Aa}{\mathscr{A}}

\newcommand{\Ya}{\mathscr{Y}}

\newcommand{\Ca}{\mathcal{C}}
\newcommand{\Xa}{\mathscr{X}}

\newcommand{\C}{\mbox{$\mathbb{C}$}}

\newcommand{\fii}{\varphi}

\newcommand{\ip}[2]{\left\langle#1,#2\right\rangle}

\usepackage{color}
\begin{document}

\setcounter{page}{1}

\title[Multi-$\Aa$-linearity preserving operators]
{Orthogonality Hilbert $\Aa$-modules and operators preserving multi-$\Aa$-linearity}
\author[P. W\'ojcik \MakeLowercase{and} A. Zamani]
{Pawe\l{} W\'ojcik$^1$ \MakeLowercase{and} Ali Zamani$^2$}

\address{$^1$Institute of Mathematics, Pedagogical University of Cracow, Podchor\c a\.zych~2, 30-084 Krak\'ow, Poland}
\email{pawel.wojcik@up.krakow.pl}
\address{$^2$Corresponding author, School of Mathematics and Computer Sciences, Damghan University, Damghan, P.~O.~BOX 36715-364, Iran}
\email{zamani.ali85@yahoo.com}

%%%%%%%%%%%%%%%%%%%%%%%%%%%%%%%%%%%%%%%%%%%%%%
\begin{abstract}
In this paper we present results concerning
orthogonality in Hilbert $C^*$-modules.
Moreover, for a $C^*$-algebra $\mathscr{A}$, we prove theorems concerning the multi-$\Aa$-linearity and its
preservation by $\Aa$-linear operators.
New version of solution of the orthogonality equation on Hilbert $C^*$-modules
and mappings preserving orthogonality are also investigated.
\end{abstract} \maketitle
%%%%%%%%%%%%%%%%%%%%%%%%%%%
\section{Introduction and preliminaries}
Let us recall some basic facts about $C^*$-algebras and
Hilbert $C^*$-modules and introduce our notation.
A \textit{$C^*$-algebra} is a complex Banach $*$-algebra $\big(\mathscr{A}, \|\!\cdot\!\|)$
such that $\|a\|^2 = \|a^*a\|$ for every $a\in \mathscr{A}$. An element $a\in \mathscr{A}$
is called \textit{positive} (we write $a\geq0$) if $a = c^*c$ for some $c\in \mathscr{A}$.
The relation $``\geq"$ on $\mathscr{A}$ is given by $a \geq b$ if and only if $a -b \geq 0$.
If $a\in \mathscr{A}$ is positive, then there exists a unique positive $p\in \mathscr{A}$ such that
$a = p^2$; such an element $p$ is called the \textit{positive square root} of $a$.
A linear functional $\varphi$ of $\mathscr{A}$ is positive if $\varphi(a)\geq0$ for every positive element
$a\in \mathscr{A}$. A \textit{state} is a positive linear functional whose norm is equal to one.
Let $\mathscr{A}$ be a $C^*$-algebra and let $\Xa$ be an algebraic \textit{right $\Aa$-module} which is a complex linear space with
compatible scalar multiplication \big(i.e., $(\alpha x)a = x(\alpha a) = \alpha (xa)$ for all $x\in \Xa, a\in \Aa, \alpha \in \C$\big).
The space $\Xa$ is called a \textit{(right) pre-Hilbert $\Aa$-module} (inner product $\Aa$-module) if there
exists an $\Aa$-valued inner product, i.e., a mapping $\ip{\cdot}{\cdot}\Xa\times \Xa\to \Aa$ satisfying

(C1)     $\ip{x}{x}\geq 0$ and $\ip{x}{x}=0$ if and only if $x = 0$,

(C2)     $\ip{x}{\alpha y+\beta z}=\alpha\ip{x}{y}+\beta\ip{x}{z}$,

(C3)     $\ip{x}{ya}=\ip{x}{y}a$,

(C4)     $\ip{x}{y}=\ip{y}{x}^*$,

for all $x, y, z\in \Xa, a\in \Aa, \alpha, \beta \in \C$.
For $x\in \Xa$, by $|x|$ we denote the unique positive square of $\ip{x}{x} \in \Aa$.
A pre-Hilbert $\Aa$-module $\Xa$ is called a \textit{Hilbert $\mathscr{A}$-module} if it is complete with
respect to the norm $\|x\|:= \sqrt{\|\ip{x}{x}\|}$.
Hilbert spaces are left Hilbert $\C$-modules.
Any $C^*$-algebra $\mathscr{A}$ is a Hilbert $C^*$-module over itself via
$\langle a, b\rangle := a^*b$. The corresponding norm is just the norm on $\mathscr{A}$
because of the $C^*$-condition.
Given a positive functional $\fii$ on $\Aa$, we have the following useful version of the Cauchy--Schwarz inequality:
\begin{align}\label{inequality-positive-functional}
\forall_{x, y\in \mathscr{X}}\ \ \ |\fii(\ip{x}{y})|^2\leq \fii(\ip{x}{x})\fii(\ip{y}{y}).
\end{align}
A Hilbert $\mathscr{A}$-module $\mathscr{X}$ is said to be \textit{full} if the two sided ideal
generated by all products $\langle x, y\rangle, \, x,y \in \mathscr{X}$,
is dense in $\mathscr{A}$. By $\Ka(\mathscr{X})$ we denote the $C^*$-algebra spanned by
$\big\{\theta_{x,y} : \, x,y \in \mathscr{X}\big\}$, where map
$\theta_{x,y}\colon\mathscr{X} \rightarrow \mathscr{X}$ is defined by $\theta_{x,y}(z):= x\langle y, z\rangle$.
We refer the reader to \cite{Lan, Mu} for more information on the basic theory of $C^{*}$-algebras and Hilbert $C^*$-modules.

A concept of orthogonality in a Hilbert $C^*$-module can be defined with respect to the
$C^*$-valued inner product in a natural way, that is, two elements $x$ and $y$ of a Hilbert $C^*$-module
$\mathscr{X}$ are called \textit{orthogonal}, in short $x \perp y$, if $\langle x, y\rangle = 0$.
Let us recall other types of orthogonality in Hilbert $C^*$-modules.
One of the most important is the concept of the Birkhoff--James orthogonality:
if $x, y$ are elements of a Hilbert $C^*$-module $\mathscr{X}$,
then $x$ is \textit{orthogonal to} $y$ \textit{in the Birkhoff--James sense} \cite{A.R.1, B.G}, in short $x \perp_{B} y$, if
\begin{align*}
\forall_{\lambda\in\mathbb{C}} \ \ \|x + \lambda y\| \geq \|x\|.
\end{align*}
As a natural generalization of the notion of Birkhoff--James orthogonality,
the concept of strong Birkhoff--James orthogonality,
which involves modular structure of a Hilbert $C^{*}$-module was introduced in \cite{A.R.2}.
When $x$ and $y$ are elements of a Hilbert $\mathscr{A}$-module $\mathscr{X}$,
$x$ is \textit{orthogonal to} $y$ \textit{in the strong Birkhoff--James sense}, in short $x \perp^s_{B} y$, if
\begin{align*}
\forall_{a\in \mathscr{A}} \ \ \|x + ya\| \geq \|x\|.
\end{align*}
This orthogonality is "between" $\perp$ and $ \perp_{B} $, i.e.,
\begin{align*}
\forall_{x,y\in \mathscr{X}}\ \ \ x \perp y \, \Longrightarrow \, x\perp^s_{B} y \, \Longrightarrow \, x\perp_{B} y,
\end{align*}
while the converses do not hold in general (see \cite{A.R.2}).

In this paper we first present results concerning
orthogonality in Hilbert $C^*$-modules.
In particular, we give a characterization of orthogonality in term of norm inequalities.
Then we prove theorems concerning the multi-$\Aa$-linearity and its
preservation by $\Aa$-linear operators. Some other related results are also discussed.
Finally, new version of solution of the orthogonality equation on Hilbert $C^*$-modules
and mappings preserving orthogonality are investigated.
%%%%%%%%%%%%%%%%%%%%%%%%%%
\section{Orthogonality in Hilbert $C^*$-modules}
Over the years many mathematicians have studied these concepts of orthogonality
in Hilbert $C^*$-modules in \cite{A.R.1, A.R.2, A.R.3, B.G, M.Z, W}.
In particular, it was proved in \cite{A.R.4, I.T, Z.M.F} that for elements $x$ and $y$ of a Hilbert $\mathscr{A}$-module $\mathscr{X}$,
$x\perp y$ is equivalent to each of the following statements:
\begin{itemize}
\item[(i)] $\forall_{\lambda \in \mathbb{C}}\ \ \ |x + \lambda y| = |x - \lambda y|$.
\item[(ii)] $\forall_{a\in \mathscr{A}}\ \ \ |x + ya| = |x - ya|$.
\item[(iii)] $\forall_{a\in \mathscr{A}}\ \ \ \|x + ya\| = \|x - ya\|$.
\item[(iv)] $\forall_{\lambda \in \mathbb{C}}\ \ \ |x + \lambda y|^2\geq |x|^2$.
\item[(v)] $\forall_{a\in \mathscr{A}}\ \ \ |x + ya|^2 \geq |x|^2$.
\item[(vi)] $\forall_{a\in \mathscr{A}}\ \ \ |x + ya| \geq |x|$.
\end{itemize}
We begin with a major motivation for stating this section.
Suppose $\mathscr{X}$ is a Hilbert $C^*$-module over a $C^*$-algebra $\mathscr{A}$ and $x, y \in \mathscr{X}$.
By the definition of the Birkhoff--James orthogonality, we have
\begin{align*}
x\perp_{B}y &\Longleftrightarrow \forall_{\lambda\in\mathbb{C}} \ \|x + \lambda y\|
\geq \|x\| \Longleftrightarrow \forall_{\alpha\in\mathbb{C}\setminus\{0\}} \ \left\|x + \frac{1}{\alpha} y\right\| \geq \|x\|
\\& \Longleftrightarrow \forall_{\alpha\in\mathbb{C}} \ \|\alpha x + y\| \geq \|\alpha x\|,
\end{align*}
and hence
\begin{align*}
x\perp_{B}y \ \Longleftrightarrow \ \forall_{\alpha\in\mathbb{C}} \ \ \|\alpha x + y\| \geq \|\alpha x\|.
\end{align*}
The above property is the key notion of the section. Namely, since in
Hilbert $C^*$-modules the role of scalars is played by the elements of the
underlying $C^*$-algebra, therefore, the question arises whether it is also true for
the strong Birkhoff--James orthogonality, i.e. whether $x\perp^s_{B}y$ is equivalent to the assertion that
\begin{align}\label{I.1.L.2.0}
\forall_{a\in \mathscr{A}}\ \ \|xa + y\| \geq \|xa\|.
\end{align}
The answer is negative in general, as shown in the following example.
\begin{example}
Let us take the $C^*$-algebra $\mathscr{A}$ of all complex $2\times 2$ matrices,
and regard it as a Hilbert $C^*$-module over itself.
Let $X = \begin{bmatrix}
1 & 0 \\
0 & 1
\end{bmatrix}$ and $Y = \begin{bmatrix}
1 & 0 \\
0 & 0
\end{bmatrix}$. Then $X\perp^s_B Y$ since for any $A = \begin{bmatrix}
a_{11} & a_{12} \\
a_{21} & a_{22}
\end{bmatrix}$
we have
\begin{align*}
\|X + YA\| = \left\|\begin{bmatrix}
1+a_{11} & a_{12} \\
0 & 1
\end{bmatrix}\right\| \geq 1 = \|X\|.
\end{align*}
But (\ref{I.1.L.2.0}) does not hold since for $A = -Y$ we have $\|XA + Y\| = 0 < \|XA\| = 1$.
\end{example}
%%%%%%%%%%%%%%%%%%%%%%%%%%%%
In the next theorem, we prove that (\ref{I.1.L.2.0}) is actually equivalent
to the orthogonality with respect to the $C^*$-valued inner product.
The following lemma, that was obtained in \cite[Theorem~2.7]{A.R.1} and \cite[Theorem~4.4]{B.G}
independently, is essential for supporting our conclusion.
%%%%%%%%%%%%%%%%%%%%%%%%
\begin{lemma}\label{L.2.0}
Let $\mathscr{X}$ be a Hilbert $\mathscr{A}$-module and $x, y \in \mathscr{X}$.
Then the following conditions are equivalent:
\begin{itemize}
\item[(i)] $x\perp_{B}y$,
\item[(ii)] There exists a state $\varphi$ on $\mathscr{A}$ such that
$\varphi(\langle x, x\rangle) = \|x\|^2$ and $\varphi(\langle x, y\rangle) = 0$.
\end{itemize}
\end{lemma}
%%%%%%%%%%%%%%%%%%%%%%%%%%%%%%%%
As another motivation for our work
in this section, recall the following results from the paper \cite{A.R.4}.
\begin{theorem}\label{theorem-arambasic-rajic-2019}\cite[Theorem 2, Corollary 1]{A.R.4}
Let $\mathscr{X}$ be a Hilbert $\mathscr{A}$-module and $x, y \in \mathscr{X}$.
Then the following conditions are equivalent:
\begin{itemize}
\item[(i)] $\forall_{a\in \mathscr{A}}\ \ \ \|x + ya\|=\|x - ya\|$,
\item[(ii)] $y\ip{y}{x}\bot_B x$,
\item[(iii)] $x\perp y$.
\end{itemize}
\end{theorem}
%%%%%%%%%%%%%%%%%%%%%%%%%%%
Now, we want to prove that the above result can be strengthen as follows.
\begin{theorem}\label{T.2.1}
Let $\mathscr{X}$ be a Hilbert $\mathscr{A}$-module and $x, y \in \mathscr{X}$.
Then the following conditions are equivalent:
\begin{itemize}
\item[(i)] $\forall_{a\in \mathscr{A}}\ \ \ |xa + y|^2 \geq |xa|^2$,
\item[(ii)] $\forall_{a\in \mathscr{A}}\ \ \ |xa + y| \geq |xa|$,
\item[(iii)] $\forall_{a\in \mathscr{A}}\ \ \ \|xa + y\| \geq \|xa\|$,
\item[(iv)] $x\perp y$.
\end{itemize}
Moreover, since the relation $\perp $ is symmetric, each of the above
conditions is equivalent to each of the three below:
\begin{itemize}
\item[(v)] $\forall_{a\in \mathscr{A}}\ \ \ |ya + x|^2 \geq |ya|^2$,
\item[(vi)] $\forall_{a\in \mathscr{A}}\ \ \ |ya + x| \geq |ya|$,
\item[(vii)] $\forall_{a\in \mathscr{A}}\ \ \ \|ya + x\| \geq \|ya\|$.
\end{itemize}
\end{theorem}
\begin{proof}
(i)$\Rightarrow$(ii) The implication follows from Theorem~2.2.6 of \cite{Mu}.

(ii)$\Rightarrow$(iii) Let $a\in \mathscr{A}$ such that $|xa + y| \geq |xa|$.
By \cite[Theorem~2.2.5~(3)] {Mu} it follows that $\big\||xa + y|\big\| \geq \big\||xa|\big\|$.
Then $\|xa + y\| \geq \|xa\|$ since $\big\||z|\big\| = \|z\|$ for all $z\in\mathscr{X}$.

(iii)$\Rightarrow$(iv) Suppose (iii) holds. Let $\lambda \in \mathbb{C}\setminus\{0\}$
and let $a= \frac{\langle x, y\rangle}{\lambda} \in \mathscr{A}$. Then
\begin{align*}
\left\|x\frac{\langle x, y\rangle}{\lambda} + y\right\| \geq \left\|x\frac{\langle x, y\rangle}{\lambda}\right\|,
\end{align*}
and so
\begin{align*}
\big\|x\langle x, y\rangle + \lambda y\big\| \geq \big\|x\langle x, y\rangle\big\|.
\end{align*}
The above inequality is also obvious for $\lambda = 0$, and therefore
\begin{align*}
\forall_{\lambda\in\begin{scriptsize}\mathbb{C}\end{scriptsize}} \ \ \big\|x\langle x, y\rangle + \lambda y\big\| \geq \big\|x\langle x, y\rangle\big\|.
\end{align*}
Thus $x\langle x, y\rangle\perp_{B}y$. By Lemma~\ref{L.2.0}, there exists a state $\varphi$ on $\mathscr{A}$ such that
\begin{align*}
\varphi\Big(\big\langle x\langle x, y\rangle, x\langle x, y\rangle\big\rangle\Big) = \big\|x\langle x, y\rangle\big\|^2
\quad \mbox{and} \quad \varphi\Big(\big\langle x\langle x, y\rangle, y\big\rangle\Big) = 0,
\end{align*}
and equivalently,
\begin{align}\label{I.1.T.2.1}
\varphi\Big(\ip{x\ip{x}{y}}{x}\ip{x}{y}\Big) = \big\|x\langle x, y\rangle\big\|^2
\quad \mbox{and} \quad \varphi\Big(\langle y, x\rangle \langle x, y\rangle\Big) = 0.
\end{align}
Hence
\begin{align*}
0 &\leq\big\|x\langle x, y\rangle\big\|^4\stackrel{\eqref{I.1.T.2.1}}{=}
\left|\varphi\Big(\ip{x\ip{x}{y}}{x}\ip{x}{y}\Big)\right|^2\\
&\stackrel{\eqref{inequality-positive-functional}}{\leq}
\varphi\Big(\ip{x\ip{x}{y}}{x}\ip{x}{x\ip{x}{y}}\Big)\varphi\Big(\ip{y}{x}\ip{x}{y}\Big)
\stackrel{\eqref{I.1.T.2.1}}{=}0.
\end{align*}
Therefore, we get $\big\|x\langle x, y\rangle\big\|^4 = 0$,
which gives $x\langle x, y\rangle = 0$.
Then we have
\begin{align*}
\big\|\langle x, y\rangle\big\|^2 =\big\|\langle y, x\rangle\langle x, y\rangle\big\|
= \big\|\big\langle y, x\langle x, y\rangle\big\rangle\big\| = \|\langle y, 0\rangle\| = 0,
\end{align*}
and so $\|\langle x, y\rangle\big\| = 0$. Thus $\langle x, y\rangle = 0$. Hence $x\perp y$.

(iv)$\Rightarrow$(i) Let $x\perp y$. Then $\langle x, y\rangle = \langle y, x\rangle = 0$.
So, for every $a\in\mathscr{A}$, we have
\begin{align*}
|xa + y|^2 = \langle xa, xa\rangle + a^*\langle x, y\rangle + \langle y, x\rangle a + \langle y, y\rangle
= |xa|^2 + |y|^2,
\end{align*}
which implies $|xa + y|^2 \geq |xa|^2$.
Thus the implication (vi)$\Rightarrow$(i) is proved.

To summarize, we may consider (i)$\Leftrightarrow$(ii)$\Leftrightarrow$(iii)$\Leftrightarrow$(iv) as
shown. Since $\bot$ is symmetric, we see
that (iv)$\Leftrightarrow$(v)$\Leftrightarrow$(vi)$\Leftrightarrow$(vii), and the proof is complete.
\end{proof}
%%%%%%%%%%%%%%%%%%%
%%%%%%%%%%%%%%%%%%%
\section{On $\Aa$-linear mappings preserving orthogonality}
%%%%%%%%%%%%%%%%%%%%%%%%
Chmieli\'nski \cite{C} had proved that a linear mapping $T\colon X\to Y$ (between inner product
spaces) which preserves orthogonality
(i.e. $x\bot y\ \Rightarrow Tx\bot Ty$ for all $x,y\in X$) has to be a
similarity (scalar multiple of an isometry).

It has been proved by Koldobsky \cite{K} (for real normed spaces) and Blanco and Turn\v{s}ek
\cite{B.T} (for real and complex ones) that a linear mapping $T\colon X\to Y$ preserving
the Birkhoff--James orthogonality, i.e., satisfying $x\bot_B y \Rightarrow Tx\bot_B Ty$, has to be a similarity.
Another proof was presented by W\'ojcik \cite{W.2}. Moreover, he
extended the Koldobsky theorem. Namely, for real spaces he showed
that the linearity assumption may be replaced
by additivity (see \cite{W.2}).

We again focus on Hilbert $C^*$-modules. So, let us consider
now $\Aa$-linear mappings $T\colon \Xa\to\Ya$ (between Hilbert $\Aa$-modules $\Xa$ and $\Ya$) preserving
an orthogonality, i.e., satisfying $x\bot y \Rightarrow Tx\bot Ty$
with respect to given $\Aa$-valued inner
products $\ip{\cdot}{\cdot}\colon \Xa\times\Xa\to\Aa$ and $\ip{\cdot}{\cdot}\colon \Ya\times\Ya\to\Aa$.

Let $\Ka(H)$ and $\Ba(H)$ denote the $C^*$-algebra of all compact operators and, respectively, the $C^*$-algebra of
all bounded operators on a Hilbert space $H$.

Before proving the main theorems
of this section, let us look at a few
classical results. Ili\v{s}evi\'{c} and Turn\v{s}ek  \cite{I.T} obtained
the following result in the case where $\Aa$ is
a standard $C^*$-algebra, that is, when $\Ka(H)\subseteq\Aa\subseteq \Ba(H)$ (see also \cite{F.M.P}).
%%%%%%%%%%%%%%%%%%%%%%%%%%%%%%%
\begin{theorem}\cite{I.T}\label{theorem1-i-t-pres-orth}
Let $\Aa$ be a standard $C^*$-algebra.
Let $\mathscr{X}$ and $\mathscr{Y}$ be Hilbert $\mathscr{A}$-modules. For a
nonzero mapping $T\colon \Xa\to\Ya$, the following assertions are equivalent:

{\rm (i)} \ $T$ is $\Aa$-linear and orthogonality preserving mapping,

{\rm (ii)} \ $T$ is a similarity, i.e., $\exists_{\gamma>0} \ \forall_{x,y\in\Xa} \ \ \ip{Tx}{Ty}=\gamma\ip{x}{y}$.
\end{theorem}
%%%%%%%%%%%%%%%%%%%%%%%%%
Recall that a linear mapping $T\colon \Xa\to\Ya$,
where $\Xa$ and $\Ya$ are Hilbert $\mathscr{A}$-modules, is called local if
\begin{align*}
\forall_{a\in\Aa} \ \forall_{x\in\Xa} \ \ xa = 0\, \Longrightarrow \,(Tx)a = 0.
\end{align*}
Note that every $\Aa$-linear mapping is local, but the converse is not true,
in general (take linear differential operators into account).
Let $\Ca_0(\Omega)$ denote the $C^*$-algebra of all continuous $\C$-valued functions
on a locally compact Hausdorff space $\Omega$, which vanish at $\infty$ with the supremum
norm. Leung, Ng and Wong \cite{L.N.W.2010} obtained
the following result in the case where $\Aa=\Ca_0(\Omega)$.
%%%%%%%%%%%%%%%%%%%%%%%%%%%%%%%
\begin{theorem}\cite{L.N.W.2010}\label{theorem2-l-n-w-pres-orth}
Let $\Omega$ be a locally compact Hausdorff space, and let $\Xa$ and $\Ya$ be two
Hilbert $\Ca_0(\Omega)$-modules. Suppose that $T\colon \Xa\to\Ya$ is an orthogonality preserving local
$\C$-linear map. The following assertions hold.

{\rm (i)} $T$ is $\Aa$-linear and orthogonality preserving mapping.

{\rm (ii)} $\forall_{x,y\in\Xa} \ip{Tx}{Ty}\!=\!\fii\ip{x}{y}$, for some bounded
nonnegative function $\fii$ on $\Omega$.
\end{theorem}
%%%%%%%%%%%%%%%%%%%%%%%%%%%%%
The same mathematicians extended their results and they
obtain another strong theorem \cite{L.N.W}. It is worth mentioning
that they did not assume again continuity of a mapping $T\colon \Xa\to\Ya$.
%%%%%%%%%%%%%%%%%%%%%%%%%%%%
\begin{theorem}\cite{L.N.W}\label{theorem3-l-n-w-pres-orth}
Let $\mathscr{X}$ and $\mathscr{Y}$ be Hilbert $\mathscr{A}$-modules. For a
nonzero mapping $T\colon \Xa\to\Ya$, the following assertions are equivalent:

{\rm (i)} \ $T$ is $\Aa$-linear and orthogonality preserving mapping,

{\rm (ii)} \ $T$ is a $\Aa$-similarity, i.e., $\exists_{c\in \Aa} \ \forall_{x,y\in\Xa} \ \ \ip{Tx}{Ty}=c\ip{x}{y}$.
\end{theorem}
%%%%%%%%%%%%%%%%%%%%%%%%%%%%
The above results have motivated our work, which is
devoted to multi-$\Aa$-linear mappings defined on the Cartesian product of Hilbert $C^*$-modules.
They will be essentially generalized.
In particular, we will replace $\Aa$-valued inner products
by multi-$\Aa$-linear functionals.

If $\Xa$ is a right $\Aa$-module, a function $F\colon \Xa^n\to\Aa$
is a {\em multi-$\Aa$-linear function} if
for $x_1,\ldots,x_j,\ldots,x_n,y_j\in \Xa$ and $\alpha\in\C$ and $a\in\Aa$ and $j=1,\ldots,n$,

\textmd{(A)} \ \ $F(x_1,\ldots,x_j+y_j,\ldots,x_n)=F(x_1,\ldots,x_j,\ldots,x_n)+F(x_1,\ldots,y_j,\ldots,x_n)$,

\textmd{(B)} \ \ $F(x_1,\ldots,\alpha x_j a,\ldots,x_n)=\alpha F(x_1,\ldots,x_j,\ldots,x_n)a$, if $j=2k$,

\textmd{(C)} \ \ $F(x_1,\ldots,\alpha x_j a,\ldots,x_n)=\overline{\alpha} a^*F(x_1,\ldots,x_j,\ldots,x_n)$, if $j=2k+1$.
\\
Now suppose, in addition, that $\Xa$ is a normed
space. A multi-$\Aa$-linear function $F\colon\Xa^n\to\Aa$ is \textit{bounded} if there is
a constant $M$ such that
\begin{center}
$\|F(x_1,\ldots,x_n)\|\leq M\!\cdot\!\|x_1\|\cdot\ldots\cdot\|x_n\|$ for all $x_1,\ldots,x_n\in\Xa$.
\end{center}

Let $\ip{\cdot}{\cdot}\colon \Xa\times\Xa\to\Aa$ and $\ip{\cdot}{\cdot}\colon \Ya\times\Ya\to\Aa$ be
$\Aa$-valued inner products.
If $T\colon \Xa\to\Ya$ is $\Aa$-linear, then a function $F\colon \Xa^2\to\Aa$ defined
by $F(x,y):=\ip{Tx}{Ty}$ is a multi-$\Aa$-linear mapping. Define $E\colon \Xa^2\to\Aa$ by $E(x,y):=\ip{x}{y}$.
Then we can rewrite the orthogonality preserving property in the form:
\begin{align}\label{pattern-orth-pres-prop}
\forall_{x,y\in\Xa}\quad E(x,y)=0\, \Longrightarrow\, F(x,y)=0.
\end{align}
Summarizing, the condition \eqref{pattern-orth-pres-prop} will be the key notion of this section.

Suppose that $\Aa$ has an identity. Put $G_{\Aa}:=\{b\in \Aa:  b\ {\rm is\ invertible}\}$. It is known
that $G_{\Aa}$ are open subsets of $\Aa$.
We say that a multi-$\Aa$-linear function $F\colon \Xa^n\to\Aa$ is \textit{strong}, if there
exists $w\in\Xa$ such that $F(w,w,\ldots,w)\in G_{\Aa}$.
Consider $\Aa'_m:=\big\{\fii\colon\Aa\to\C\, |\, \fii\ {\rm is\ linear\ and\ multiplicative}\big\}$.
We include the following useful property about multiplicative functionals. This basic fact
perhaps appears elsewhere but we present its with proof for sake of completeness.
%%%%%%%%%%%%%%%%%%%%%%%
\begin{lemma}\label{theorem-tot-impl-ab}
Let $\Aa$ be a $C^*$-algebra. If the set $\Aa'_m$ is total, i.e.,
\begin{center}
$\left(\forall_{\fii\in \Aa'_m} \ \ \fii(a)=0\right) \Longrightarrow a=0$,
\end{center}
then $\Aa$ is abelian.
\end{lemma}
\begin{proof}
Fix $a, b\in\Aa$ and $\fii\in\Aa'_m$. Then $\fii(ab - ba)=\fii(a)\fii(b) -\fii(b)\fii(a)=0$.
Since $\Aa'_m$ is total, $ab - ba=0$. Thus $ab = ba$.
\end{proof}
%%%%%%%%%%%%%%%%%%%%%%%%%
Finally, we are ready to present that Theorems \ref{theorem1-i-t-pres-orth}, \ref{theorem2-l-n-w-pres-orth} and
\ref{theorem3-l-n-w-pres-orth} can be extended for the case that $\Aa$ is abelian.
These will not be only shallow and cosmetic changes.
We will show new methods to tackle this kind of
theorem. So, we state our first main result of this section.
%%%%%%%%%%%%%%%%%%%%%%%%%%%%%%%%%%%%%
\begin{theorem}\label{main-2-result-form-pres}
Let $\Aa$ be a $C^*$-algebra with identity. Assume that $\Aa'_m$ is
a total set. Let $\mathscr{X}$ be a right $\mathscr{A}$-module, and suppose that $\Xa$ be a normed space.
Let $E\colon \Xa^2\to\Aa$ and $F\colon \Xa^2\to\Aa$ be
multi-$\Aa$-linear functions. Moreover, assume that $E$ is bounded and strong.
The following assertions are equivalent:

{\rm (i)} \ $\forall_{x, y\in\Xa}\quad E(x, y)=0\, \Longrightarrow\, F(x, y)=0$,

{\rm (ii)} \ $\exists_{c\in\Aa}\ \forall_{x, y\in\Xa}\quad F(x, y)=cE(x, y)$.\\
Moreover, each of the above conditions implies:

{\rm (iii)} $F$ is bounded.
\end{theorem}
\begin{proof}
The implication (ii)$\Rightarrow$(i) is obvious. We show (i)$\Rightarrow$(ii).
Suppose (i) holds. We may assume without loss of generality that $E$ is a bounded multi-$\Aa$-linear function with bound
$M=1$, i.e. $\|E(x, y)\|\leq\|x\|\!\cdot\!\|y\|$ for all $x, y\in\Xa$.

Since $E$ is strong, there exists $w\in\Xa$ such that $E(w,w)\in G_{\Aa}$. It follows
easily that $w\neq 0$. Put $p:=E(w,w)$.
Since $G_{\Aa}$ is open, there is an open ball
$B(p,r)\subseteq\Aa$ such that $B(p,r) \subseteq G_{\Aa}$ with some $r>0$.
Now there is a small number $\beta>0$ such that $(\beta+2\|w\|)\cdot\beta< r$.
Let us consider an open
ball $B\left(w,\beta\right)\subseteq\Xa$. Fix $x\in B\left(w,\beta\right)$.
Thus we have
\begin{align*}
\| E(x,x)-p\|&=\| E(x,x) - E(w,w)\|\\
&=\| E(x,x)-E(x,w)+E(x,w) - E(w,w)\|\\
&= \|E(x,x\!-\!w)+E(x\!-\!w,w)\|\\
&\leq \|E(x,x\!-\!w)\|+ \|E(x\!-\!w,w)\|\\
&\leq\|x\|\cdot\|x-w\|+\|x-w\|\cdot\|w\|=(\|x\|+\|w\|)\cdot\|x-w\|\\
&\leq(\|x-w\|+\|w\|+\|w\|)\cdot\|x-w\|<(\beta+2\|w\|)\cdot\beta< r.
\end{align*}
Hence $E(x,x)\in B(p,r) \subseteq G_{\Aa}$.
In other words, we proved the following property:
\begin{align}\label{property-invertibility}
x\in B\left(w,\beta\right) \ \ \Longrightarrow \ \ E(x,x) \ {\rm is\ invertible}.
\end{align}
Thus, we may define $h\colon B\left(w,\beta\right)\to\Aa$ by $h(z):=F(z,z)E(z,z)^{-1}$.

Fix $y\in \Xa$ and fix $z\in B\left(w,\beta\right)$. Since $E(z,z)$ is invertible, we may
define
\begin{align*}
b:=-E(z,z)^{-1}E(z,y)\quad \mbox{and} \quad d:=\left(-E(z,z)^{-1}E(y,z)\right)^*.
\end{align*}
It is clear
that $E(z,zb+y)=0$ and $E(zd+y,z)=0$. From (i) it follows
that $F(z,zb+y)=0$ and $F(zd+y,z)=0$. We thus
get $F(z,y)=-F(z,z)b$ and
similarly $F(y,z)=-d^*F(z,z)$. Lemma \ref{theorem-tot-impl-ab} shows
that $\Aa$ is abelian. Hence
\begin{align}\label{FE-equality-1}
\begin{array}{ll}F(z,y)=F(z,z)E(z,z)^{-1}E(z,y)
& {\rm and}\\
F(y,z)=F(z,z)E(z,z)^{-1}E(y,z)& {\rm for}\ z\in B\left(w,\beta\right), \ y\in\Xa.
\end{array}
\end{align}
Since $B\left(w,\beta\right)$ is open, it follows that there is $t_o>0$ such
that $z+ty\in B\left(w,\beta\right)$ for
all $t\in (-t_o,t_o)$. Hence, by \eqref{property-invertibility}, $E(z+ty,z+ty)$ is invertible
for all $t\in (-t_o,t_o)$.
So, fix a number $t\in (-t_o,t_o)\setminus\{0\}$. Therefore, we may write
\begin{align*}
h(z+ty)-h(z)&=F(z+ty,z+ty)E(z+ty,z+ty)^{-1}-F(z,z)E(z,z)^{-1}\\
&=E(z+ty,z+ty)^{-1}E(z,z)^{-1}\\
&\quad\quad\Big(E(z,z)\big(F(z+ty,z+ty)-F(z,z)\big)\Big.\\
&\quad\quad\quad -\Big.F(z,z)\big(E(z+ty,z+ty)-E(z,z)\big)\Big)\\
&=E(z+ty,z+ty)^{-1}E(z,z)^{-1}\\
&\quad\quad\Big(E(z,z)\big(tF(y,z)+tF(z,y)+t^2F(y,y)\big)\Big.\\
&\quad\quad\quad -\Big.F(z,z)\big(tE(y,z)+tE(z,y)+t^2E(y,y)\big)\Big).
\end{align*}
Fix $\fii\in\Aa'_m$. It follows from the above equalities that
\begin{align*}
\frac{\fii(h(z+ty))-\fii(h(z))}{t}&=\fii\Big(E(z+ty,z+ty)^{-1}E(z,z)^{-1}\Big)\\
&\quad\quad\Big[\fii\Big(E(z,z)\big(F(y,z)+F(z,y)+tF(y,y)\big)\Big)\\
&\quad\quad\quad -\fii\Big(F(z,z)\big(E(y,z)+E(z,y)+tE(y,y)\big)\Big)\Big].
\end{align*}
From this it may be conclude that
\begin{align*}
\lim\limits_{t\rightarrow 0}\frac{\fii(h(z+ty))-\fii(h(z))}{t}&=\fii\Big(E(z,z)^{-1}E(z,z)^{-1}\Big)\\
&\quad\quad\Big[\fii\Big(E(z,z)\big(F(y,z)+F(z,y)\big)\Big)\\
&\quad\quad\quad -\fii\Big(F(z,z)\big(E(y,z)+E(z,y)\big)\Big)\Big].
\end{align*}
Since $\fii$ is multiplicative (and linear), we have
\begin{align*}
\lim\limits_{t\rightarrow 0}\frac{\fii(h(z+ty))-\fii(h(z))}{t}&=\fii\big(E(z,z)^{-1}\big)\fii\Big(F(z,y)-E(z,z)^{-1}F(z,z)E(z,y)\Big)\\
&+\fii\big(E(z,z)^{-1}\big)\fii\Big(F(y,z)-E(z,z)^{-1}F(z,z)E(y,z)\Big)\\
&\stackrel{\eqref{FE-equality-1}}{=}\fii\big(E(z,z)^{-1}\big) 0+\fii\big(E(z,z)^{-1}\big) 0=0.
\end{align*}
From this we conclude that $\fii\circ h$ is a constant
function. Since $\Aa_m'$ is total, $h\colon B\left(w,\beta\right)\to\Aa$ is a constant
function, too.
So, we define $c:=h(\cdot)$.
Putting $c$ in place of $F(z,z)E(z,z)^{-1}$ in the first equality \eqref{FE-equality-1} we get
\begin{align}\label{FE-equality-2}
F(z,y)=cE(z,y)\ {\rm for} \ z\in B\left(w,\beta\right), \ y\in\Xa.
\end{align}

Fix $y\in \Xa$. Let $L_y\colon\Xa\to\Aa$ be a linear
operator defined by the formula
\begin{align*}
L_y(\cdot):=F(\cdot,y)-cE(\cdot,y).
\end{align*}
It is worth mentioning that the continuity of $L_y$ is not needed; more precisely:
we do not know that $F$ is bounded.
The condition \eqref{FE-equality-2} says that $L_y$ is constant on
the open set $B\left(w,\beta\right)$. Therefore the linear operator $L_y$ is constant
on the whole space $\Xa$.
Now it is clear that $F(z,y)-cE(z,y)=0$ for all $z\in\Xa$.
Since $y$ was arbitrarily chosen from $\Xa$, we
obtain $F(z,y)-cE(z,y)=0$ for all $z,y\in\Xa$. But now we know that $F$ is bounded.
\end{proof}
%%%%%%%%%%%%%%%%%%%%%%%%%%
The above results are similar to that of \cite{W-2015}. But
the investigation \cite{W-2015} was easier, because the
paper \cite{W-2015} contains only the cases $\Aa=\C$ and $n=2$.
%%%%%%%%%%%%%%%%%%%%%%%%
\begin{theorem}\label{main-n-result-form-pres}
Let $\Aa$ be a $C^*$-algebra with identity. Assume that $\Aa'_m$ is a total set.
Let $\mathscr{X}$ be a right $\mathscr{A}$-module, and suppose that $\Xa$ be a normed space.
Let $n\geq 2$ and let $E\colon \Xa^n\to\Aa$ and $F\colon \Xa^n\to\Aa$ be
multi-$\Aa$-linear functions. Moreover, assume that $E$ is bounded and strong.
The following assertions are equivalent:

{\rm (i)} \ $\forall_{x_1,\ldots,x_n\in\Xa}\quad E(x_1,\ldots,x_n)=0\, \Longrightarrow\, F(x_1,\ldots,x_n)=0$,

{\rm (ii)} \ $\exists_{c\in\Aa}\ \forall_{x_1,\ldots,x_n\in\Xa}\quad F(x_1,\ldots,x_n)=cE(x_1,\ldots,x_n)$.\\
Moreover, each of the above conditions implies:

{\rm (iii)} $F$ is bounded.
\end{theorem}
\begin{proof}
The proof of Theorem \ref{main-n-result-form-pres} is by induction. For $n=2$ we have
proved that it is true. Indeed, see Theorem \ref{main-2-result-form-pres}.
Suppose that Theorem \ref{main-n-result-form-pres} holds for $n$. Now, assume
that $E\colon \Xa^{n+1}\to\Aa$, $F\colon \Xa^{n+1}\to\Aa$ are
multi-$\Aa$-linear functions. Moreover, assume that $E$ is bounded and strong. Suppose that
\begin{center}
$\forall_{x_1,\ldots,x_n,x_{n+1}\in\Xa}\quad E(x_1,\ldots,x_n,x_{n+1})=0\, \Longrightarrow\, F(x_1,\ldots,x_n,x_{n+1})=0$.
\end{center}
Since $E$ is strong, it follows that there
exists $w\in\Xa$ with $E(w,\ldots,w,w)\in G_{\Aa}$.
It is clear that $w\neq 0$. Put $p:=E(w,\ldots,w,w)$.
We know that $G_{\Aa}$ is open and $E$ is bounded. Analysis similar to that
in the proof of Theorem \ref{main-2-result-form-pres} shows that there is
an open ball $B(w,\beta)\subseteq\Xa$ such that
\begin{align}\label{property-invertibility-n}
x\in B\left(w,\beta\right) \ \ \Longrightarrow \ \ E(x,\ldots,x,x)\in G_{\Aa},
\end{align}
i.e. $E(x,\ldots,x,x)$ is invertible. Fix $b\in B(w,\beta)$. The continuity of $E$ implies
that there is an open ball $B(u,\eta)\subseteq B(w,\beta)$ such that
\begin{align}\label{prop-e-g-a}
\begin{array}{ll}
\forall_{z\in B(u,\eta)}\ \forall_{y\in\Xa}\ \exists_{t_{z,y}>0}\ \forall_{|t|<t_{z,y}}&
z+ty\in B(u,\eta),\ \ {\rm and}\\
&E\left(\stackrel{(1)}{z},\ldots,\stackrel{(n-1)}{z},\stackrel{(n)}{z+ty},\stackrel{(n+1)}{z+ty}\right)\in G_{\Aa}.
\end{array}
\end{align}
Define multi-$\Aa$-linear functions
$E_b\colon \Xa^{n}\to\Aa$ and $F_b\colon \Xa^{n}\to\Aa$ by
the following formulas
\begin{align*}
E_b(x_1,\ldots,x_n):=E(x_1,\ldots,x_n,b) \quad \mbox{and} \quad F_b(x_1,\ldots,x_n):=F(x_1,\ldots,x_n,b).
\end{align*}
It is easily seen that
\begin{center}
$\forall_{x_1,\ldots,x_n\in\Xa}\quad E_b(x_1,\ldots,x_n)=0\, \Longrightarrow\, F_b(x_1,\ldots,x_n)=0$,
\end{center}
and $E_b$ is strong by \eqref{property-invertibility-n}.
It follows from the induction hypothesis that there
exists a function $\widehat{c}\colon\ B(w,\beta)\to\Aa$ such that
\begin{align}\label{property-c-from-induct}
\forall_{x_1,\ldots,x_n\in\Xa}\quad F_b(x_1,\ldots,x_n)=\widehat{c}(b)E_b(x_1,\ldots,x_n).
\end{align}
Now the property \eqref{property-c-from-induct} becomes
\begin{align}\label{property-c-from-induct-b}
\forall_{x_1,\ldots,x_n\in\Xa}\quad F(x_1,\ldots,x_n,b)=\widehat{c}(b)E(x_1,\ldots,x_n,b).
\end{align}
It is helpful to recall
that $z+ty\in B(w,\beta)$. Combining \eqref{property-c-from-induct-b}
with \eqref{prop-e-g-a} yields
\begin{center}
$\widehat{c}(z+ty)=F(z,\ldots,z,z+ty,z+ty)E(z,\ldots,z,z+ty,z+ty)^{-1}$.
\end{center}
Fix $\fii\in\Aa'_m$.
In a similar way as in the proof of Theorem \ref{main-2-result-form-pres} we
obtain $\lim\limits_{t\rightarrow 0}\frac{\fii\left(\widehat{c}(z+ty)\right)-\fii\left(\widehat{c}(z)\right)}{t}=0$.
Since $\Aa_m'$ is total, $\widehat{c}\colon B\left(w,\beta\right)\to\Aa$ is constant on $B(u,\eta)$; so we
define $c:=\widehat{c}(\cdot)|_{B(u,\eta)}$. Now the property \eqref{property-c-from-induct-b} becomes
\begin{align*}
\forall_{x_1,\ldots,x_n\in\Xa}\ \forall_{z\in B(u,\eta)}\quad F(x_1,\ldots,x_n,z)=cE(x_1,\ldots,x_n,z).
\end{align*}
Thus a linear operator
$F(x_1,\ldots,x_n,\cdot)-cE(x_1,\ldots,x_n,\cdot)\colon\Xa\to\Aa$
is constant
on the open set $B(u,\eta)$. Therefore
$F(x_1,\ldots,x_n,x_{n+1})-cE(x_1,\ldots,x_n,x_{n+1})=0$ for all $x_j\in\Xa$ with $j=1,\ldots,n,n+1$ (see also the
proof of Theorem \ref{main-2-result-form-pres}).
\end{proof}
%%%%%%%%%%%%%%%%%%%%%
Combining Theorem~\ref{T.2.1} with Theorem \ref{main-2-result-form-pres},
we obtain characterizations of the orthogonality preserving property for pairs of
mappings on Hilbert $C^*$-modules (see also \cite{A.O, C.L.W, F.M.P, F.M.Z, I.T, L.N.W, M.Z}).
%%%%%%%%%%%%%%%%%%%
\begin{corollary}\label{C.2.4}
Let $\Aa$ be a $C^*$-algebra with identity.
Let $\mathscr{X}$ and $\mathscr{Y}$ be Hilbert $\mathscr{A}$-modules and $x, y \in\mathscr{X}$.
For mappings $T, S: \mathscr{X} \rightarrow \mathscr{Y}$
the following assertions are equivalent:
\begin{itemize}
\item[(i)] $x\perp y \Longrightarrow Tx \perp Sy$,
\item[(ii)]
$\forall_{a\in \mathscr{A}}\ |xa + y|^2 \geq |xa|^2 \Longrightarrow \forall_{a\in \mathscr{A}}\ \big|(Tx)a + Sy\big|^2 \geq \big|(Tx)a\big|^2$,
\item[(iii)]
$\forall_{a\in \mathscr{A}}\ |xa + y| \geq |xa| \Longrightarrow \forall_{a\in \mathscr{A}}\ \big|(Tx)a + Sy\big| \geq \big|(Tx)a\big|$,
\item[(vi)]
$\forall_{a\in \mathscr{A}}\ \|xa + y\| \geq \|xa\|\Longrightarrow \forall_{a\in \mathscr{A}}\ \big\|(Tx)a + Sy\big\| \geq \big\|(Tx)a\big\|$.
\end{itemize}
Moreover, if $\Aa'_m$ is total, then each of the above conditions is equivalent to:
\begin{itemize}
\item[(v)] $\exists_{c\in\Aa}\, \forall_{x,y\in\Xa} \ip{Tx}{Sy}=c\ip{x}{y}$.
\end{itemize}
If, in addition, $T=S$, then each of the above conditions is equivalent to:
\begin{itemize}
\item[(vi)] $\forall_{x,y\in\Xa}\ |x|\leq|y|\,\Longrightarrow\, |Tx|\leq|Ty|$.
\end{itemize}
\end{corollary}
%%%%%%%%%%%%%%%%%%%%%%
Next we show that the some multi-$\Aa$-linear mapping preserve the invertibility in the below sense.
%%%%%%%%%%%%%%%%%%%%
\begin{corollary}\label{main-n-result-form-pres-inv}
Let $\Aa$ be a $C^*$-algebra with identity. Assume that $\Aa'_m$ is a total set.
Let $\mathscr{X}$ be a right $\mathscr{A}$-module, and suppose that $\Xa$ be a normed space.
Let $n\geq 2$ and let $E\colon \Xa^n\to\Aa$ and $F\colon \Xa^n\to\Aa$ be
multi-$\Aa$-linear functions. Moreover, assume that $E$ is bounded and strong.
Suppose that
\begin{align}\label{prop-pres-orth-inv-n}
\forall_{x_1,\ldots,x_n\in\Xa}\quad E(x_1,\ldots,x_n)=0\, \Longrightarrow\, F(x_1,\ldots,x_n)=0.
\end{align}
The following assertions are equivalent:

{\rm (i)} \ $\exists_{z\in\Xa}\ E(z,\ldots,z),F(z,\ldots,z)\in G_{\Aa}$,

{\rm (ii)} \ $\forall_{z_1,\ldots,z_n\in\Xa}\ E(z_1,\ldots,z_n)\in G_{\Aa}\, \Longrightarrow\, F(z_1,\ldots,z_n)\in G_{\Aa}$.
\end{corollary}
\begin{proof}
It is visible that (i) implies (ii). Now, we assume (ii). Combining \eqref{prop-pres-orth-inv-n} and (ii) and
Theorem \ref{main-n-result-form-pres}, we immediately get
$F=cE$. Since $E$ is strong, the equality $F=cE$ implies that $c\in G_{\Aa}$. The rest is clear.
\end{proof}
%%%%%%%%%%%%%%%%%%%%
%%%%%%%%%%%%%%%%%%%%
\section{Concluding remarks}
Let us quote a result from \cite{A.R.3}.
%%%%%%%%%%%%%%%%%%%%
\begin{lemma}\cite[Theorem~2.6]{A.R.3}\label{L.2.5}
Let $\mathscr{X}$ be a full Hilbert $\mathscr{A}$-module.
Then the following statements are equivalent:
\begin{itemize}
\item[(i)] $\perp_{B}$ is a symmetric relation,
\item[(ii)] $\perp^s_{B} \,= \,\perp$,
\item[(iii)] $\mathscr{A}$ or $\Ka(\mathscr{X})$ is isomorphic to $\mathbb{C}$.
\end{itemize}
\end{lemma}
%%%%%%%%%%%%%%%%%
Our next result is stated as follows.
%%%%%%%%%%%%%%%%%
\begin{theorem}\label{T.2.6}
Let $\mathscr{X}$ be a full Hilbert $\mathscr{A}$-module.
Then the following statements are equivalent:
\begin{itemize}
\item[(i)]  $\forall_{x,y\in\Xa}\ \ x\perp_{B}y \Longrightarrow \|x + y\|\geq \|y\|$,
\item[(ii)] $\mathscr{A}$ or $\Ka(\mathscr{X})$ is isomorphic to $\mathbb{C}$.
\end{itemize}
\end{theorem}
\begin{proof}
(i)$\Rightarrow$(ii) Suppose (i) holds.
Let $z$ and $w$ be arbitrary elements of $\mathscr{X}$ such that $z\perp^s_{B}y$.
Then $z\perp_{B}wa$ for all $a\in \mathscr{A}$.
So our assumption implies that $\|wa + z\| \geq \|wa\|$ for all $a\in \mathscr{A}$.
By Theorem~\ref{T.2.1} it follows $z\perp w$. Thus, $\perp^s_{B} \subseteq \perp$ and hence
$\perp^s_{B}$ coincides with the inner product orthogonality.
By the equivalence (ii)$\Rightarrow$(iii) in Lemma~\ref{L.2.5}, we deduce that
$\mathscr{A}$ or $\Ka(\mathscr{X})$ is isomorphic to $\mathbb{C}$.

(ii)$\Rightarrow$(i) If (ii) holds then, by the equivalence (i)$\Rightarrow$(iii) in Lemma~\ref{L.2.5},
$\perp_{B}$ is a symmetric relation. Now, let $x, y \in \mathscr{X}$ be such that $x\perp_{B}y$.
Then we get $y\perp_{B}x$. In particular (taking $\lambda = 1$), we conclude $\|x + y\|\geq \|y\|$.
\end{proof}
%%%%%%%%%%%%%%%%%%
When a $C^*$-algebra $\mathscr{A}$ is regarded as a Hilbert $\mathscr{A}$-module,
then $\Ka(\mathscr{A})$ is isomorphic to $\mathscr{A}$ (see e.g. \cite[p.10]{Lan}).
Therefore, as a consequence of Theorem~\ref{T.2.6}, we have the following result.
%%%%%%%%%%%%%%%%%%
\begin{corollary}\label{C.2.7}
Let $\mathscr{A}$ be a $C^*$-algebra.
The following statements are equivalent:
\begin{itemize}
\item[(i)]  $\forall_{a, b\in\Aa}\ \ a\perp_{B}b \Longrightarrow \|a + b\|\geq \|b\|$,
\item[(ii)] $\mathscr{A}$ is isomorphic to $\mathbb{C}$.
\end{itemize}
\end{corollary}
It remains an unanswered question whether Theorem \ref{main-n-result-form-pres} is
valid also for non-abelian algebras. We show the following affirmative answer, but only for $n=1$.
\begin{theorem}\label{main-1-result-form-pres}
Let $\Aa$ be a $C^*$-algebra with identity (not necessary abelian).
Let $\mathscr{X}$ be a
right $\mathscr{A}$-module. Let $E\colon \Xa\to\Aa$ and $F\colon \Xa\to\Aa$ be
$\Aa$-linear functions. Moreover, assume that $E$ is strong (not necessary bounded). The following
assertions are equivalent:

{\rm (i)} \ $\forall_{x\in\Xa}\quad E(x)=0\, \Longrightarrow\, F(x)=0$,

{\rm (ii)} \ $\exists_{c\in\Aa}\ \forall_{x\in\Xa}\quad F(x)=cE(x)$.
\end{theorem}
\begin{proof}
Clearly it is enough to prove (i)$\Rightarrow$(ii). Assume (i).
Since $E$ is strong,
there
exists $w\in\Xa$ such that $E(w,w)\in G_{\Aa}$. Put $c=F(w)E(w)^{-1}$. Fix $x\in\Xa$. It is easy
to see that $E\big(-wE(w)^{-1}E(x)+x\big)=0$. Hence $F\big(-wE(w)^{-1}E(x)+x\big)=0$, and
so $F(x)=F(w)E(w)^{-1}E(x)=cE(x)$.
\end{proof}
%%%%%%%%%%%%%%%%%%%%%
In the context of orthogonality in an inner product space $X$, the classical functional equation
\begin{align}\label{orth-eq-classical-version}
\forall_{x,y\in X}\ \ \langle Tx, Ty\rangle=\langle x,y\rangle
\end{align}
corresponding to orthogonal transformations has received
a lot of contributions in the literature; see e.g. \cite{RL.PW} and \cite{I.T}.
One of them is below and it motivates our next result.
%%%%%%%%%%%%%%%%%%%%%%%%%%%%%%
\begin{theorem}\cite[Proposition 2.3]{I.T}\label{ilisevic-tunsek-theorem-a}
Let $\mathscr{X}$ and $\mathscr{Y}$ be inner product $\mathscr{A}$-modules. For a
mapping $T\colon\mathscr{X}\! \to\! \mathscr{Y}$ and for some $\gamma\!\in\! (0,+\infty)$ the
following assertions are equivalent:
\begin{itemize}
\item[(i)]  $\forall_{x,y\in \Xa}\ \ \langle Tx, Ty\rangle=\gamma^2\langle x,y\rangle$,
\item[(ii)] $T$ is $\mathscr{A}$-linear and $\|Tx\|=\gamma\|x\|$\  {\rm for\ all}\ $x\in  \mathscr{X}$.
\end{itemize}
\end{theorem}
%%%%%%%%%%%%%%%%%%%%%%%%
Now, we consider a natural generalization of an orthogonality
equation \eqref{orth-eq-classical-version} again on inner product $C^*$-modules.
More precisely, we want to show that Theorem \ref{ilisevic-tunsek-theorem-a} can be strengthen as follows.
%%%%%%%%%%%%%%%%%%%%%%%%%%%%
\begin{theorem}\label{prop-il-tur-inner-pr-pres}
Let $\mathscr{X}, \mathscr{Y}$ be inner product $\mathscr{A}$-modules.
Suppose that $T\colon\mathscr{X} \to \mathscr{Y}$ be a nonvanishing mapping. Let $\gamma\in (0,+\infty)$.
Then, the following conditions are equivalent:
\begin{itemize}
\item[(i)] $\forall_{x,y\in \Xa}\ \ \langle Tx, Ty\rangle=\gamma^2\langle x,y\rangle$,
\item[(ii)] $T$ is $\mathscr{A}$-linear and $|Tx|=\gamma|x|$ for all $x\in\mathscr{X}$,
\item[(iii)] $T$ is linear and $|Tx|=\gamma|x|$ for all $x\in\mathscr{X}$.
\end{itemize}
\end{theorem}
Before launching into the proof, a few words motivating the proof are
appropriate. It is easy to see that (ii)$\Rightarrow$(iii) is trivial. The $\mathscr{A}$-linearity implies
linearity, but the reverse is not true. Therefore the
the implication (iii)$\Rightarrow$(ii) seems interesting.
\begin{proof}
Taking into account Theorem \ref{ilisevic-tunsek-theorem-a}, it remains to
prove only (iii)$\Rightarrow$(i). For two fixed
vectors $x,y\in \mathscr{X}$, we can derive the following chain of equalities
\begin{align*}
\langle Tx,Ty\rangle &=\! \frac{1}{4}\left( |Ty\!+\!Tx|^2\!-\!|Ty\!-\!Tx|^2 \!-\!i |iTy\!+\!Tx|^2\!+\!i|iTy\!-\!Tx|^2\right)\\
&=\frac{1}{4}\left( |T(y\!+\!x)|^2\!-\!|T(y\!-\!x)|^2 \!-\!i |T(iy\!+\!x)|^2\!+\!i|T(iy\!-\!x)|^2\right)\\
&=\frac{1}{4}\gamma^2\left(|y\!+\!x|^2\!-\!|y\!-\!x|^2 \!-\!i |iy\!+\!x|^2\!+\!i|iy\!-\!x|^2\right)\\
&=\gamma^2\langle x,y\rangle,
\end{align*}
which means that $T$ satisfies the property (i), and we are done.
\end{proof}
%%%%%%%%%%%%%%%%%%%%%%%%%%%%%%%%%%%%%%%%%%%
%\textbf{Acknowledgement.}
%The authors would like to thank the referee for her/his valuable suggestions and comments.
%%%%%%%%%%%%%%%%%%%%%%%%%%%%%%%%%%%%%%%%%%%
\bibliographystyle{amsplain}

\begin{thebibliography}{99}

\bibitem{A.R.1} Lj.~Aramba\v{s}i\'{c} and R.~Raji\'{c},
\textit{The Birkhoff--James orthogonality in Hilbert $C^*$-modules},
Linear Algebra Appl. \textbf{437} (2012), 1913--1929.

\bibitem{A.R.2} Lj.~Aramba\v{s}i\'{c} and R.~Raji\'{c},
\textit{A strong version of the Birkhoff--James orthogonality in Hilbert $C^*$-modules},
Ann.~Funct.~Anal. \textbf{5} (2014), no. 1, 109--120.

\bibitem{A.R.3} Lj.~Aramba{\v{s}}i\'c and R.~Raji\'c,
\textit{On symmetry of the (strong) Birkhoff--James orthogonality in Hilbert $C^*$-modules}
Ann.~Funct.~Anal. \textbf{7} (2016), no. 1, 17--23.

\bibitem{A.R.4} Lj.~Aramba{\v{s}}i\'c and R.~Raji\'c,
\textit{Another characterization of orthognality in Hilbert C*-modules},
Math.~Inequal.~Appl. \textbf{22} (2019), no. 4, 1421--1426.

\bibitem{A.O} M.~B.~Asadi and F.~Olyaninezhad,
\textit{Orthogonality preserving pairs of operators on Hilbert $C_0(Z)$-modules},
Linear Multilinear Algebra, doi: 10.1080/03081087.2020.1825610.

\bibitem{B.G} T.~Bhattacharyya and P.~Grover,
\textit{Characterization of Birkhoff--James orthogonality},
J. Math. Anal. Appl. \textbf{407} (2013), no. 2, 350-–358.

\bibitem{B.T} A.~Blanco and A.~Turn\v{s}ek,
\textit{On maps that preserve orthogonality in normed spaces},
Proc.~Roy.~Soc.~Edinburgh Sect.~A \textbf{136} (2006) 709--716.

\bibitem{C} J.~Chmieli\'nski,
\textit{Linear mappings approximately preserving orthogonality},
J. Math. Anal. Appl.  \textbf{304} (2005), 158--169.

\bibitem{C.L.W} J.~Chmieli\'{n}ski, R.~{\L}ukasik and P.~W\'{o}jcik,
\textit{On the stability of the orthogonality equation and the orthogonality-preserving property with two unknown functions},
Banach J.~Math.~Anal. \textbf{10} (2016), no. 4, 828--847.

\bibitem{F.M.P} M.~Frank, A.S.~Mishchenko and A.A.~Pavlov,
\textit{Orthogonality-preserving, $C^*$-conformal and conformal module mappings on Hilbert $C^*$-modules},
J.~Funct.~Anal. \textbf{260} (2011), 327--339.

\bibitem{F.M.Z} M.~Frank, M.S.~Moslehian and A.~Zamani,
\textit{Orthogonality preserving property for pairs of operators on Hilbert $C^*$-modules},
Aequationes Math. doi: 10.1007/s00010-021-00790-1.

\bibitem{I.T} D.~Ili\v{s}evi\'{c} and A.~Turn\v{s}ek,
\textit{Approximately orthogonality preserving mappings on $C^*$-modules},
J.~Math.~Anal.~Appl. \textbf{341} (2008), 298--308.

\bibitem{K} A.~Koldobsky,
\textit{Operators preserving orthogonality are isometries},
Proc.~Roy.~Soc.~Edinburgh Sect.~A \textbf{123} (1993), 835--837.

\bibitem{Lan} E.~C.~Lance,
\textit{Hilbert $C^*$-modules. A Toolkit for Operator Algebraists},
London Mathematical Society Lecture Note Series, vol. 210, Cambridge University Press, Cambridge, 1995.


\bibitem{L.N.W.2010} C.-W.~Leung, C.-K.~Ng and N.-C.~Wong,
\textit{Linear orthogonality preservers of Hilbert bundles},
J.~Aust.~Math.~Soc. \textbf{89} (2010), 245--254.


\bibitem{L.N.W} C.-W.~Leung, C.-K.~Ng and N.-C.~Wong,
\textit{Linear orthogonality preservers of Hilbert $C^*$-modules},
J.~Operator Theory \textbf{71} (2014), no. 2, 571--584.

\bibitem{RL.PW}
R.~\L ukasik and P.~W\'{o}jcik, \textit{Decomposition of two functions in the orthogonality equation},
Aequationes Math. \textbf{90} (2016), 495--499.


\bibitem{M.Z} M.~S.~Moslehian and A.~Zamani,
\textit{Mappings preserving approximate orthogonality in Hilbert $C^*$-modules},
Math Scand. \textbf{122} (2018), 257--276.

\bibitem{Mu} G.J.~Murphy,
\textit{$C^*$-Algebras and Operator Theory},
Academic Press, New York, 1990.


\bibitem{W-2015} P.~W\'{o}jcik,
\textit{Operators preserving sesquilinear form},
Linear Algebra Appl. \textbf{469} (2015), 531--538.


\bibitem{W} P.~W\'{o}jcik,
\textit{The Birkhoff orthogonality in pre-Hilbert $C^*$-modules},
Oper.~Matrices \textbf{10} (2016), no. 3, 713--729.


\bibitem{W.2} P.~W\'{o}jcik,
\textit{Mappings preserving $B$-orthogonality},
Indag.~Math. (N.S.) \textbf{30} (2019) 197--200.


\bibitem{Z.M.F} A.~Zamani, M.S.~Moslehian and M.~Frank,
\textit{Angle preserving mappings},
Z.~Anal.~Anwend. \textbf{34} (2015), 485--500.
\end{thebibliography}

\end{document}